\documentclass[leqno,12pt]{article}
\usepackage{amsmath}
\usepackage{amssymb,amscd}
\usepackage{amsthm}
\usepackage{latexsym}
\usepackage[affil-it]{authblk}

\usepackage{color}  

\newtheorem{thm}{Theorem}[section]

\newtheorem{prop}[thm]{Proposition}
\newtheorem{cor}[thm]{Corollary}

\newtheorem{ex}[thm]{Example}
\newtheorem{defin}[thm]{Definition}

\newcommand{\du}{S \! \Join^b \! E}

\title{Minimal genus of a multiple and Frobenius number of a quotient of a numerical semigroup} 
\author{ F. Strazzanti 
}

\date{}

\AtEndDocument{\bigskip\footnotesize
\textsc{Department of Mathematics, University of Pisa, Largo Bruno Pontecorvo 5, 56127 Pisa, Italy} \\
\textit{E-mail address}: \texttt{strazzanti@mail.dm.unipi.it} 
}

\makeatletter

\def\blfootnote{\xdef\@thefnmark{}\@footnotetext}

\makeatother

\begin{document}
\maketitle

\blfootnote{\noindent Preprint of an article published in International Journal of Algebra and Computation}
\blfootnote{ {\bf 25} (2015), no. 6, 1043--1053, DOI: 10.1142/S0218196715500290} 
\blfootnote{$\copyright$ World Scientific Publishing Company} \blfootnote{http://www.worldscientific.com/doi/abs/10.1142/S0218196715500290 }

\begin{abstract}
\noindent
Given two numerical semigroups $S$ and $T$ and a positive integer $d$, $S$ is said to be one over $d$ of $T$ if $S=\{s \in \mathbb{N} \ | \ ds \in T \}$ and in this case $T$ is called a $d$-fold of $S$.
We prove that the minimal genus of the $d$-folds of $S$ is $g + \lceil \frac{(d-1)f}{2} \rceil$, where $g$ and $f$ denote the genus and the Frobenius number of $S$. The case $d=2$ is a problem proposed by Robles-P\'erez, Rosales, and Vasco.
Furthermore, we find the minimal genus of the symmetric doubles of $S$ and study the particular case when $S$ is almost symmetric.
Finally, we study the Frobenius number of the quotient of some families of numerical semigroups. 

\medskip


Symmetric numerical semigroup $\cdot$ Almost symmetric semigroup $\cdot$ \\
Frobenius number $\cdot$ $d$-symmetric semigroup $\cdot$ Numerical duplication.
\end{abstract}

\section{Introduction}

A monoid $S \subseteq \mathbb{N}$ is called {\it numerical semigroup} if $\mathbb{N} \setminus S$ is finite.

Numerical semigroup theory has application in several contexts. 
For example in \cite{RGSGU} Rosales, Garc\'ia-S\'anchez, Garc\'ia-Garc\'ia,
and Urbano-Blanco studied Diophantine inequalities of the form
$ax \mod b \leq cx$ with $a, b$ and $c$ positive integers 
and proved that the set of their nonnegative solutions 
is a numerical semigroup; they called such a semigroup {\it proportionally modular numerical semigroup}.
In order to characterize these semigroups, they introduced the concept of {\it quotient} of a numerical semigroup: 
given two numerical semigroups $S,T$ and a positive integer $d$, we write $S=\frac{T}{d}$ if $S=\{s \in \mathbb{N} \ | \ sd \in T\}$;
we also say that $T$ is a {\it multiple} of $S$.
In \cite[Corollary 3.5]{RR} it is proved that a numerical semigroup is proportionally modular if and only if there exist
two positive integers $a,d$ such that $S=\frac{\langle a,a+1 \rangle}{d}$.

After these works, the notion of quotient of a numerical semigroup has been widely studied. 
In \cite{RRV} Robles-P\'erez, Rosales, and Vasco studied the set of the doubles of a numerical semigroup $S$,
that is the set of the semigroups $T$ such that $S=\frac{T}{2}$. At the end of their article, the authors asked for a formula, that depend on $S$, for computing the minimum of the genus of a double of $S$, where the {\it genus} of a numerical semigroup is the cardinality of its complement in $\mathbb{N}$. 
In this paper we find this formula in a more general frame; indeed in Theorem \ref{main} we prove that 
$$
\min \left\{g(\overline S) \ | \ S=\frac{\overline{S}}{d} \right\}= g(S) + \left\lceil \frac{(d-1)f(S)}{2} \right\rceil,
$$
where $g(\cdot)$ denotes the genus and $f(\cdot)$ is the {\it Frobenius number}, that is the largest integer that is not in the numerical semigroup.

If $R$ is a one-dimensional, noetherian, local, and analytically irreducible domain, it is possible to associate to it a numerical semigroup $S$. In this case the properties of $R$ can be translated in numerical properties and conversely;
a famous result in this context is that $R$ is Gorenstein if and only if $S$ is symmetric (see Chapter II of \cite{BDF}). For this reason symmetric numerical semigroups are extensively studied and in this context, in \cite{BF}, Barucci and Fr\"{o}berg introduced a generalization of symmetric semigroups and Gorenstein rings that they called, respectively, {\it almost symmetric numerical semigroups} and {\it almost Gorenstein rings}.

Rosales and Garc\'ia-S\'anchez proved in \cite{RG} that, given a numerical semigroup $S$, there exist infinitely many symmetric doubles of $S$. We also find the minimal genus of the semigroups in this family and construct the unique semigroup that has this genus; to do this, we use the {\it numerical duplication} of a numerical semigroup defined in \cite{DS}. The formula is more satisfactory when $S$ is almost symmetric.

A numerical semigroup is said to be $d${\it -symmetric } if for all integers $n \in \mathbb Z$, whenever $d$ divides $n$, either $n$ or $f(S)-n$ is in $S$. In the last section we prove that, if $S$ is $d$-symmetric, then 
$$
f \left( \frac{S}{d} \right) = \frac{f(S)-x}{d}, 
$$ 
where $x$ is the minimum integer of $S$ such that $x \equiv f(S) \mod d$. As a corollary, this formula holds for symmetric and pseudo-symmetric numerical semigroups.
Unfortunately, in general, it is difficult to find $x$, but, when $S$ is symmetric and $d=2$, it is simply the minimum odd generator of $S$.
We also give a description of $f(\frac{S}{5})$ when $S$ is generated by two consecutive elements.
This research is motivated by the characterization of proportionally modular numerical semigroups and it is an open question proposed first in \cite{RG1} and after in \cite{DGR}. 

Finally, another simple formula is found for some particular almost symmetric semigroups (cfr. Corollary \ref{quoziente simmetrico}).
\vspace{1em}

The paper is organized as follows. In the second section we find the formula for the minimal genus of a multiple of a numerical semigroup $S$ (see Theorem \ref{main}) and in Corollary \ref{almost symmetric} we characterize the semigroup with minimal genus when $S$ is almost symmetric. In Section $3$ we find the minimal genus of a symmetric double of $S$ and construct the unique semigroup that has this genus (see Proposition \ref{symmetric}); moreover, in Corollary \ref{almost symmetric 2} we study the particular case of $S$ almost symmetric. Finally, in the last section, we prove the formula for the Frobenius number of the quotient of a $d$-symmetric numerical semigroup (see Theorem \ref{Frobenius quoziente}) and we give corollaries in some particular cases.

\section{Minimal genus of a multiple of a numerical semigroup}

Given a numerical semigroup $S$, we call the elements of $\mathbb N \setminus S$ {\it gaps} and their cardinality $g(S)$ is the {\it genus} of $S$; moreover the maximum integer that is not in $S$ is called {\it Frobenius number} of $S$ and it is denoted by $f(S)$.

Let $S$ and $T$ be two numerical semigroups. We say that $S$ is one over $d$ of $T$ or that $T$ is a $d$-fold of $S$,
if $S=\{s \in \mathbb{N} \ | \ ds \in T\}$. If $d=2$ we say that $S$ is one half of $T$ or that 
$T$ is a double of $S$ and we denote the set of doubles of $S$ with ${\rm D}(S)$.
Robles-P\'erez, Rosales, and Vasco proposed the following problem:
\vspace{1em}

\textbf{Problem} \cite[Problem 42]{RRV} Let $S$ be a numerical semigroup. Find a \\
formula, that depend on $S$, for computing 
$\min \{g(\overline S) \ | \ \overline{S} \in {\rm D}(S) \}$.
\vspace{1em}

In this paper we answer to this question in a more general frame; 
indeed in the next theorem we find a formula for the minimal genus of a $d$-fold of $S$.
We set $d \cdot X = \{dx \ | \ x \in X \}$.

\begin{thm} \label{main}
Let $S$ be a numerical semigroup with Frobenius number $f$ and let $d \geq 2$ be an integer. Then 
$$
\min \left\{g(\overline S) \ | \ S=\frac{\overline{S}}{d} \right\}= g(S) + \left\lceil \frac{(d-1)f}{2} \right\rceil.
$$
\end{thm}

\begin{proof}
We first note that $\lceil \frac{(d-1)f}{2} \rceil = \lfloor \frac{df}{2} \rfloor - \lfloor \frac{f}{2} \rfloor$.
Let $\overline{S}$ be a $d$-fold of $S$. 
It is easy to see that $\overline{S}$ has exactly $g(S)$ gaps that are multiples of $d$. 
In order to count the other gaps, we note that $df \notin \overline{S}$ because $f \notin S$.
Consider $0<x \leq df/2$, $x$ not a multiple of $d$. 
If $x$ is not a gap of $\overline{S}$, then $df-x \notin \overline{S}$, since $df$ is not in $\overline{S}$.
All these gaps are different and, since there are $\lfloor \frac{f}{2} \rfloor$ multiples of $d$ smaller than $df/2$, 
they are $\lfloor \frac{df}{2} \rfloor - \lfloor \frac{f}{2} \rfloor$.
Hence we get $g(\overline{S}) \geq g(S) + \lfloor \frac{df}{2} \rfloor - \lfloor \frac{f}{2} \rfloor = g(S) + \lceil \frac{(d-1)f}{2} \rceil$.

To prove the theorem we only need to exhibit a $d$-fold of $S$ with genus 
$g(S) + \lfloor \frac{df}{2} \rfloor - \lfloor \frac{f}{2} \rfloor$. 
It easy to see that such a semigroup is
$$
T= d \cdot S \cup \{b + i \ | \ i \in \mathbb{N}, \ d \nmid b+i\},
$$
where $b= \lfloor \frac{df}{2} \rfloor +1$.
\end{proof}

In the next example we will see that there can be several numerical semigroups
whose genus equals the minimum obtained in the previous theorem. 
However, by its proof, it is clear that all these semigroups have Frobenius number $df$;
consequently there are finitely many such semigroups.
If $x_1, \dots, x_r$ are integers, we denote by $\{x_1, \dots, x_r \rightarrow \}$ the set 
$\{x_1, \dots, x_r\} \cup \{x_r+i \ | \ i \in \mathbb{N} \}$.

\begin{ex} \rm
Consider $S=\{0,3,6 \rightarrow\}$ and set $d=3$.
According to the proof of the previous theorem, a ``triple'' of $S$
with minimal genus is \\
$\{0,8,9,10,11,13,14,16 \rightarrow\}$
that has genus $g(S) + \lceil \frac{(3-1)5}{2} \rceil = 4+5=9$.

Notice that $\{0,7,9,10,11,13,14,16 \rightarrow\}$
is another triple of $S$ with minimal genus.
\end{ex}

The integers $x \notin S$ such that $x+s \in S$ for all $s \in S \setminus \{0\}$ are called {\it pseudo-Frobenius numbers} and we denote by ${\rm PF}(S)$ the set of the pseudo-Frobenius numbers of $S$. The integer $t(S)=|{\rm PF}(S)|$ is the 
{\it type} of $S$. Moreover we call a numerical semigroup $S$ {\it almost symmetric} if ${\rm L}(S) \subseteq {\rm PF}(S)$,
where ${\rm L}(S)=\{x \in \mathbb{Z} \setminus S \ | \ f(S)-x \notin S \}$. 

Given a numerical semigroup $S$, it is always true (see \cite[Proposition 2.2]{N}) that
$$
g(S) \geq \frac{f(S)+t(S)}{2}
$$
and it is well-known that the equality holds if and only if $S$ is {\it almost symmetric}.

\begin{prop}
Let $S$ be a numerical semigroup with Frobenius number $f$ and
let $\overline{S}$ be a $d$-fold of $S$ with minimal genus. \\
1) If $f$ is even or $d$ is odd, $\overline{S}$ has type $t(S)$; \\
2) If $f$ is odd and $d$ is even, $\overline{S}$ has type either $t(S)$ or $t(S)+1$. Moreover if $S$ is almost symmetric, $\overline{S}$ has type $t(S)+1$. 

\end{prop}

\begin{proof}
By the proof of Theorem \ref{main} 
it is clear that if $x \notin \overline{S}$ is not a multiple of $d$,
then $df-x \in \overline{S}$ or $x=df/2$. 
Consequently the pseudo-Frobenius numbers of $\overline{S}$, 
different from $df/2$, are multiples of $d$. 
It is easy to see that this implies $t(\overline{S}) \leq t(S)$, 
if $f$ is even or $d$ is odd, and $t(\overline{S}) \leq t(S)+1$ otherwise,
because in the first case $df/2$ is not an integer or it is a multiple of $d$.

Let $x$ be a pseudo-Frobenius number of $S$, we claim that $dx \in {\rm PF}(\overline S)$.
Clearly if $dy \in \overline{S}$, then $dx+dy=d(x+y) \in \overline{S}$, since $x \in {\rm PF}(S)$. 
Let $z$ be a nonzero element of $\overline{S}$ that is not a multiple of $d$ and suppose that $dx+z \notin \overline S$.
We have already noted that $df - dx -z=d(f-x)-z \in \overline{S}$, then $d(f-x)=d(f-x)-z+z \in \overline{S}$ 
and this implies $f-x \in S$; this is a contradiction because $x \in {\rm PF}(S)$ and then $x+(f-x)=f \in S$. 
Hence $t(\overline{S}) \geq t(S)$ and this proves 1) and the first part of 2).

We only need to prove that if $f$ is odd, $d$ is even, and $S$ is almost symmetric,
then $df/2$ is a pseudo-Frobenius number of $\overline{S}$.
Suppose that there exists $x \in \overline{S} \setminus \{0\}$ such that $df/2 +x \notin \overline S$.
If $df-(df/2+x)=df/2-x \in \overline S$, then 
$df/2 \in \overline S$ and thus $df \in \overline{S}$; contradiction. 
If $df/2-x \notin \overline S$, since it is a multiple of $d$ by the beginning of the proof, we have $\frac{df/2-x}{d} \in {\rm L}(S) \subseteq {\rm PF}(S)$ 
that implies $df/2-x \in {\rm PF}(\overline S)$ by the first part of the proof
and thus $df/2=df/2-x+x \in \overline S$; contradiction. 
\end{proof}

If $S$ is not almost symmetric, $f$ is odd, and $d$ is even, the type of $\overline{S}$ can be $t(S)$ as shows the following example. However we note that in this case the semigroup $T$ constructed in the proof of Theorem \ref{main} always has type $t(S)+1$.

\begin{ex} \rm
Consider $S=\{0,5,6,7,10 \rightarrow \}$ that is not almost symmetric. 
We note that the numerical semigroups
$$
T=\{0,19,20,21,22,23,24,25,26,27,28,29,30,31,33,34,35,37 \rightarrow\},
$$
$$
T'=\{0,14,19,20,21,23,24,25,26,27,28,29,30,31,33,34,35,37 \rightarrow\}
$$
are $4$-folds of $S$ with minimal genus. 
Since $18$ is a pseudo-Frobenius number of $T$ but not of $T'$, 
we have $t(T)=3=t(S)+1$ and $t(T')=2=t(S)$.

\end{ex}

\begin{cor} \label{almost symmetric}
Let $S$ be an almost symmetric numerical semigroup with Frobenius number $f$ and let $\overline{S}$ be a $d$-fold of $S$.
Set $t=t(S)$ if $f$ is even or $d$ is odd, and set $t=t(S)+1$ otherwise.
Then $\overline{S}$ has minimal genus if and only if it is almost symmetric with type $t$.
\end{cor}

\begin{proof}

In light of the previous proposition we can assume that $\overline{S}$ has type $t$. 
Suppose first that $f$ is even or $d$ is odd. Since $S$ is almost symmetric, $t(S)=2g(S)-f$; then 
$$
g(\overline{S}) \geq \frac{df+t}{2} = \frac{df+2g(S)-f}{2} = g(S)+\frac{(d-1)f}{2}=g(S)+\left\lceil \frac{(d-1)f}{2}\right\rceil
$$
and we have the equality, i.e. $\overline{S}$ has minimal genus, if and only if $\overline{S}$ is almost symmetric. In the second case we can use the same argument.
\end{proof}

Since a $d$-fold of $S$ with minimal genus has Frobenius number $df(S)$, we get the next corollary.

\begin{cor} \label{quoziente simmetrico}
Let $S, \overline{S}$, and $t$ be as in the previous corollary. If $\overline S$ is almost symmetric with type $t$, then $d$ divides $f(\overline{S})$ and $f(S)=f(\overline{S})/d$.
\end{cor}

In the last section we will study the Frobenius number of the quotient for other families of numerical semigroups.

\section{Genus of a symmetric double of a numerical semigroup}

A numerical semigroup $S$ is called {\it symmetric} if it has type one.
It is well-known that $S$ is symmetric if and only if $g(S)=\frac{f(S)+1}{2}$ (see \cite[Corollary 4.5]{RG1})
and this implies that a symmetric numerical semigroup is almost symmetric.

In \cite{RG} it is proved that every numerical semigroup is one half of infinitely many symmetric semigroups.
We set 
$$
\mathcal{D}(S)=\left\{\overline{S} \text{ \ symmetric numerical semigroup} \ | \ S=\frac{\overline{S}}{2} \right\}.
$$ 
In this section we want to find the elements of $\mathcal{D}(S)$ with minimal genus.

Let $b$ be an odd element of $S$. The {\it numerical duplication} of a numerical semigroup $S$ with respect to a relative ideal $E$ and $b$ is defined in \cite{DS} as 
$$
\du = 2 \cdot S \cup (2 \cdot E + b)
$$
where we recall that $2 \cdot X=\{ 2x \ | \ x \in X \}$. This is a numerical semigroup if and only if $E+E+b \subseteq S$. 
Since $S$ is one half of $\du$, this construction is useful to study the doubles of $S$ 
(for another application of numerical duplication see \cite{BDS}). 

Set $K=\{k \ | \ f(S)-k \notin S \}$; we call $E$ {\it canonical ideal} of $S$ if $E=K+x$ for some $x \in \mathbb{Z}$. In \cite[Corollary 3.10]{S} it is proved that 
$$
\mathcal{D}(S)= \{ \du \ | \ E+E+b \subseteq S \text{ \ and } E \text{ is a canonical ideal of } S \}.
$$ 
Set $E=K+x$, then it is straightforward to check that 
$\du = S \! \Join^{b'} \! K$, where $b'=2x+b \in S$. Consequently we get

$$
\mathcal{D}(S)= \{ S \! \Join^b \! K \ | \ K+K+b \subseteq S \}.
$$
Let $S \! \Join^b \! K$ be an element of $\mathcal{D}(S)$. 
By construction and since $f(K)=f(S)$, its Frobenius number is the maximum between $2f(S)$ and $2f(S)+b$, that is $2f(S)+b$.
Since $S \! \Join^b \! K$ is symmetric, its genus is 
$$
g(S \! \Join^b \! K)=\frac{f(S \! \Join^b \! K)+1}{2}=\frac{2f(S)+b+1}{2}=f(S)+\frac{b+1}{2},
$$ 
therefore we get the minimum genus when $b$ is minimum. Hence we can state the following proposition:

\begin{prop} \label{symmetric}
There exists a unique numerical semigroup with minimal genus among the members of $\mathcal{D}(S)$ and
it is $S \! \Join^b \! K$, where $b$ is the minimum odd element of $S$ such that
$K+K+b \subseteq S$. Its genus is $f(S)+\frac{b+1}{2}$.
\end{prop}

We notice that if $b>f(S)$, then $K+K+b \subseteq S$; this means that we have to check a finite number 
of elements to find the $b$ with the property of the above proposition.

\begin{ex} \rm
Consider the numerical semigroup $S=\{0,5,7,8,10,12 \rightarrow \}$. 
In this case $K=\{0,2,5,7,8,9,10,12 \rightarrow \}$. Since $2+2+5 \notin S$
and $0+2+7 \notin S$, we have $K+K+5 \nsubseteq S$ and $K+K+7 \nsubseteq S$,
while $K+K+13 \subseteq S$ because $13$ is greater than $f(S)$.
Hence the symmetric double of $S$ with minimal genus is
$S \! \Join^{13} \! \! K \! = \! \{0,10,13,14,16,17,20,23,24,26,27,28,29,30,31,32,33,34,36 \rightarrow \}$
that has genus $f(S)+\frac{b+1}{2}=18$. 
Note that the minimum genus of a double of $S$ is $g(S)+ \lceil \frac{f(S)}{2} \rceil=13$
and, according to the proof of Theorem \ref{main}, it is obtained by the semigroup 
$\{0,10,13,14,15,16,17,19,20,21,23 \rightarrow \}$.
\end{ex}

\noindent
There is a particular case in which we can avoid looking for the ``right'' $b$.

\begin{cor} \label{almost symmetric 2}
Let $S$ be an almost symmetric numerical semigroup. 
The element with minimal genus in the family $\mathcal{D}(S)$ is 
$S \! \Join^b \! K$, where $b$ is the minimum odd element of $S$.
Its genus is $f(S)+ \frac{b+1}{2}$.
\end{cor} 

\begin{proof}
It is enough to prove that $K+K+b \subseteq S$.
By definition $K=S \cup {\rm L}(S)$ and this implies $K=(S \cup {\rm PF}(S)) \setminus \{f(S)\}$
because ${\rm L}(S) = {\rm PF}(S) \setminus \{f(S) \}$, since $S$ is almost symmetric.
Therefore it is very easy to check that \\
$K+K+b \subseteq S$, since $b \in S \setminus \{0\}$.
\end{proof}

\begin{ex} \rm
Consider the almost symmetric numerical semigroup \\
$S \!= \!\{0,6,7,11,12,13,14,16 \rightarrow \}$.
Since $7$ is the minimum odd integer of $S$, the minimal genus
of a symmetric double of $S$ is $15+8/2=19$ and it is obtained by
$\{0,7,12,14,17,19,21,22,24,26,27,28,29,31,32,33,34,35,36,38 \rightarrow \}$.
\end{ex}

\section{The Frobenius number of the quotient}

In this section we study the Frobenius number of a quotient of some particular numerical semigroups.
In a particular case we have already found a formula in Corollary \ref{quoziente simmetrico}.

The next definition is due to Swanson, see \cite{Sw}. 

\begin{defin}
Let $d$ be a positive number. A numerical semigroup is said to be $d$-symmetric if for all integers $n \in \mathbb Z$, whenever $d$ divides $n$, either $n$ or $f(S)-n$ is in $S$.
\end{defin}

It is easy to see that a symmetric numerical semigroup is $d$-symmetric for any $d \in \mathbb{N}$ and that a $1$-symmetric semigroup is symmetric (see \cite[Proposition 4.4]{RG1}).

\begin{thm} \label{Frobenius quoziente}
Let $d \geq 2$ be an integer and let $S$ be a $d$-symmetric numerical semigroup. 
If $x$ is the smallest element of $S$ such that $x \equiv f(S) \ ({\rm mod} \ d)$, then
$$
f \left( \frac{S}{d} \right) = \frac{f(S)-x}{d}. 
$$ 
\end{thm}

\begin{proof}
Since $x \in S, f(S)-x$ is a gap of $S$; therefore $\frac{f(S)-x}{d} \notin \frac{S}{d}$.
Suppose by contradiction that there exists a gap $y$ of $\frac{S}{d}$ greater than $\frac{f(S)-x}{d}$.
Then $dy$ is not in $S$ and this implies $f(S)-dy \in S$, since $S$ is $d$-symmetric. Consequently we have $x \leq f(S)-dy$ by definition of $x$, but this is a contradiction because $y > \frac{f(S)-x}{d}$, i.e. $x > f(S)-dy$.  
\end{proof}

\begin{ex} \rm
Consider the numerical semigroup $S = \{0,6,9,10,12,14 \rightarrow \}$ that is $3$-symmetric but not $4$-symmetric (and then not symmetric). By the previous theorem, we get
$$
f\left(\frac{S}{3} \right)=\frac{13-10}{3}=1,
$$ 
while for $\frac{S}{4}$ the formula gives again $1$ that is not the Frobenius number of $S/4$.
\end{ex}

We remember that a numerical semigroup is {\it pseudo-symmetric} if $f(S)$ is even and for any $n \notin S$ different from $f(S)/2$, one has $f(S)-n \in S$.

\begin{cor} \label{Frobenius symmetric}
Let $S$ be either a symmetric or a pseudo-symmetric numerical semigroup. 
If $x$ is the smallest element of $S$ such that $x \equiv f(S) \mod d$, then
$$
f \left( \frac{S}{d} \right) = \frac{f(S)-x}{d}
$$
for any integer $d \geq 2$.
\end{cor}

\begin{proof}
We note that a symmetric numerical semigroup is $d$-symmetric, 
while a pseudo-symmetric numerical semigroup is $d$-symmetric if and only if $2d$ 
does not divide $f(S)$. If $f(S)$ is a multiple of $d$ we have $x=0$ and it is trivial to note that
$f(S)/d$ is the maximum gap of $\frac{S}{d}$.
\end{proof}

Since symmetric and pseudo-symmetric numerical semigroups are the almost symmetric numerical semigroups with type one and two respectively, it is natural to ask if the previous corollary holds also for almost symmetric numerical semigroups.
The next example shows that this is not true.

\begin{ex} \rm
Let $d \geq 2$ be an integer and consider the numerical semigroup
$S=\{0, d+2 \rightarrow \}$.
We have $\frac{S}{d}=\{0,2 \rightarrow \}$ but the formula of the previous corollary would give 
$$f\left( \frac{S}{d} \right)=\frac{d+1-(2d+1)}{d}=-1$$
that is false.
\end{ex}

If $a_1, \dots, a_e \in \mathbb{N}$ we set 
$$
\langle a_1, \dots, a_e \rangle = \{\lambda_1 a_1 + \dots +\lambda_e a_e \ | \ \lambda_1, \dots, \lambda_e \in \mathbb{N}\};
$$
this is a numerical semigroup if and only if $\gcd(a_1, \dots, a_e)=1$. 
We say that $a_1, \dots, a_e$ is a system of generators of $\langle a_1, \dots, a_e \rangle$; it is well-known that there exists a unique minimal system of generators of a numerical semigroup.

Unfortunately, in general, it is not easy to find the element $x$ of Theorem \ref{Frobenius quoziente}, but we can say more in a particular case.

\begin{cor} \label{Frobenius}
Let $S$ be a symmetric numerical semigroup. Then
$$
f \left( \frac{S}{2} \right) = \frac{f(S)-x}{2},
$$
where $x$ is the smallest odd generator of $S$.
\end{cor}

\begin{proof}
The thesis follows from the previous corollary because in a symmetric numerical semigroup the Frobenius number is odd
(see \cite[Proposition 4.4]{RG1}).
\end{proof}

We note that the previous corollary holds for any system of generators.
The next corollary was proved in a different way in \cite[Proposition 7]{R}. 

\begin{cor}
Let $a<b$ be two positive integers with $\gcd(a,b)=1$. Then
$$
f\left(\frac{\langle a,b \rangle}{2}\right)= 
\begin{cases}
    \frac{ab-b}{2} - a  & {\rm \ if} \ a \ {\rm is \ odd} , \\
    \frac{ab-a}{2} - b  & {\rm \ if} \ a \ {\rm is \ even} .
  \end{cases}
$$
\end{cor}

\begin{proof}
It is well-known that $\langle a,b \rangle$ is symmetric, see Corollary 4.17 of \cite{RG1}, 
moreover Sylvester proved in \cite{Sy} that $f(\langle a,b \rangle)=ab-a-b$ (it is possible to find this result also in \cite[Proposition 2.13]{RG1}).
If $a$ is odd, we have $x=a$ and then
$$
f\left(\frac{\langle a,b \rangle}{2}\right)=\frac{f(\langle a,b \rangle)-a}{2}=\frac{ab-a-b-a}{2}=\frac{ab-b}{2}-a.
$$
If $a$ is even, then $b$ is odd because $\gcd(a,b)=1$ and consequently $x=b$. 
The second formula can be found with the same argument of the first one.  
\end{proof}

In \cite{RG1} and in \cite{DGR} the authors ask for a formula for $f\left(\frac{\langle a,b \rangle}{d}\right)$, at least if $b=a+1$.
To this aim it is possible to use Corollary \ref{Frobenius symmetric}, but can be difficult to find $x$. 
In the next corollary we give a formula for $d=5$.

\begin{cor}
Let $a$ be a positive integer. Then
 
$$
f\left(\frac{\langle a,a+1 \rangle}{5}\right)= 
\begin{cases}
    \frac{a^2}{5} - a - 1  & {\rm \ if} \ a \equiv 0  \mod 5,  \\
    \frac{a^2-3a-3}{5} & {\rm \ if} \ a \equiv 1,2  \mod 5, \\
    \frac{a^2-a-1}{5} & {\rm \ if} \ a \equiv 3  \mod 5, \\
    \frac{a^2-1}{5} - a  & {\rm \ if} \ a \equiv 4  \mod 5.
  \end{cases}
$$
\end{cor}

\begin{proof}
By Sylvester's formula the Frobenius number of $\langle a, a+1 \rangle$ is $a^2-a-1$. 
We only need to find the minimum $x \in S$ such that $x \equiv a^2-a-1 \mod 5$ and it is easy to see that
$$
\begin{array}{ll}
x=4(a+1) & {\rm \ if} \ a \equiv 0  \mod 5, \\
x=2(a+1) & {\rm \ if} \ a \equiv 1,2  \mod 5, \\
x=0 & {\rm \ if} \ a \equiv 3  \mod 5, \\
x=4a & {\rm \ if} \ a \equiv 4  \mod 5.
\end{array}
$$
\end{proof}




Clearly Corollary \ref{Frobenius} gives a formula for $f\left(\frac{S}{2} \right)$, provided that we know a formula for $f(S)$. In the next corollary we collect some cases using the formulas found in \cite[Corollary 3.11]{H}, \cite[Theorem 4]{F}, and \cite[Th\'{e}or\`{e}me 2.3]{BC} respectively; for the first and the third case see also Remark 10.7 and Proposition 9.15 of \cite{RG1}.

Let $T=\langle n_1, \dots, n_e \rangle$ be a numerical semigroup. We set
$$
x=\min \{n_i \ | \ n_i \text{ is odd} \},
$$
$$
c_i= \min \{k \in \mathbb{N}\setminus \{0\} \ | \ kn_i \in \langle n_1, \dots,n_{i-1}, n_{i+1}, \dots, n_e \rangle \},
$$
$$
c_i n_i = \sum_{j \neq i} r_{ij}n_j.
$$

\begin{cor}
1) Let $S$ be a numerical semigroup with three minimal generators. It is symmetric if and only if
$S= \langle am_1, am_2, bm_1 + cm_2 \rangle$, where $a,b,c,m_1,m_2$ are natural numbers such that $a \geq 2, b+c \geq 2, \gcd (m_1,m_2)=1$, and $\gcd(a,bm_1+cm_2)=1$. In this case 
$$
f\left(\frac{S}{2}\right) = \frac{a(m_1 m_2 - m_1 - m_2)+(a-1)(bm_1 + cm_2)-x}{2};
$$   
2) Let $S=\langle n_1, \dots, n_4 \rangle$ be a symmetric numerical semigroup that is not complete intersection, then
$$
f\left(\frac{S}{2}\right)= \frac{n_2 c_2 + n_3c_3 + n_4r_{14} - (n_1+n_2+n_3+n_4+x)}{2};
$$     
3) If $S$ is a free numerical semigroup for the arrangement of its minimal generators $\{n_1, \dots, n_e \}$, then
$$
f\left(\frac{S}{2}\right)=\frac{(c_2 - 1)n_2 + \dots + (c_e -1)n_e -n_1 - x}{2}.
$$ 
\end{cor}

\vspace{1em}

\noindent \textbf{Acknowledgments.} The author thanks Pedro Garc\'ia-S\'anchez for his careful reading of an earlier version of the paper.

\end{document}